\documentclass[11pt,twoside]{article}
\usepackage{amsmath, amssymb, amsfonts, amstext, amsthm, textcomp, enumerate}
\usepackage[mathscr]{euscript}
\usepackage{float}
\usepackage{booktabs}
\usepackage{mathtools}
\usepackage{graphicx}
\usepackage{caption}
\usepackage{epstopdf}
\usepackage{longtable}
\usepackage[utf8]{inputenc}
\usepackage{color}
\usepackage{hyperref}
\usepackage{dcolumn}
\usepackage{bm}
\usepackage[english]{babel}
\usepackage{subfigure}
\usepackage{xcolor}
\usepackage{ulem}
\usepackage{titlesec}
\titleformat{\section}
  {\normalfont\Large\bfseries\centering}
  {\thesection.}{1em}{}

\newtheorem{thm}{Theorem}[section]
\newtheorem{lem}[thm]{Lemma}
\newtheorem{cor}[thm]{Corollary}
\newtheorem{pro}[thm]{Proposition}

\newtheorem{rem}[thm]{Remark}
\newtheorem{ex}[thm]{Example}

\newfont{\bb}{msbm10}

\title{Bounds on the Minimum Eigenvalue Modulus for Hadamard Products of $\mathbf{M}$- and $\mathbf{H}$-Matrices and Their Inverses}

\author{
Bharat Pratap Chauhan\thanks{
Department of Mathematics,
Indian Institute of Technology Gandhinagar,
Gujarat 382355, India. (bharat024pratap@gmail.com, bc102@snu.edu.in)
}
\and
Samir Mondal\thanks{
Department of Mathematics and Statistics,
University of Regina,
Regina, SK, Canada. (Corresponding Author : isamirmondal@gmail.com, ma19d750@smail.iitm.ac.in)
}
\and
Sushmitha P\thanks{
Department of Mathematics,
Indian Institute of Technology Patna,
Bihta 801106, India. (sushmitha@iitp.ac.in)
}
}

\date{}
\begin{document}
\maketitle

\begin{abstract}
The quantity $q(A\circ A^{-1})$, the minimum modulus of the eigenvalues of $A\circ A^{-1}$, arises naturally in connection with positive diagonal symmetrizability. For an invertible $\mathbf{M}$-matrix $A$ of order $n$, the classical bounds $\frac{2}{n}\leq q(A\circ A^{-1})\leq 1$ are known. We discuss the sharpness of the lower bound $\frac{2}{n}$ and investigate the converse of a related result involving the Jacobi iteration matrix. In particular, we show that $\rho(J_{A_k})\to 1$ does not, in general, imply $q(A_k\circ A_k^{-1})\to \frac{2}{n}$, and identify a class for which this implication holds.

We then turn to invertible $\mathbf{H}$-matrices, a broader class that contains invertible $\mathbf{M}$-matrices. We show that $A\circ A^{-1}$ is an invertible $\mathbf{H}$-matrix whenever $A$ is an invertible $\mathbf{H}$-matrix. In contrast to the $\mathbf{M}$-matrix setting, $q(A\circ A^{-1})$ can be arbitrarily close to zero. However, replacing $A^{-1}$ by the inverse of the comparison matrix restores the classical lower bound: we prove that $q(A\circ\mathcal{M}(A)^{-1})\geq \frac{2}{n}$ and obtain further bounds involving the Jacobi iteration matrix of $\mathcal{M}(A)$.  Finally, for positive diagonally symmetrizable invertible $\mathbf{H}$-matrices, we establish the upper bound $q(A\circ A^{-1})\leq1$ and, in the irreducible case, characterize when equality occurs.
\end{abstract}

\vspace{0.5cm}
\vspace{0.5cm}

\section{Introduction and Background}\label{sec:introduction}

The Hadamard product, also known as the Schur product, is one of the fundamental entrywise operations in matrix analysis. Given two matrices \(A=(a_{ij})\) and \(B=(b_{ij})\) of the same size, their Hadamard product is defined by
$
A\circ B=(a_{ij}b_{ij}).
$
The classical Schur product theorem \cite{schur1911} states that the Hadamard product of two positive semidefinite matrices is again positive semidefinite. This fundamental preservation property has motivated the study of Hadamard products for several important classes of structured matrices. Of particular relevance to the present work, Lynn \cite{lynn1964} established that the class of \(\mathbf{H}\)-matrices is closed under the Hadamard product. Hadamard products involving \(\mathbf{M}\)-matrices and inverse \(\mathbf{M}\)-matrices have led to a number of spectral inequalities; see, for example, \cite{kcsmanami,chen2004alower,fiedler1988aninequality,johnson1977ahadamard,song2000onaninequality}.

The class of \(\mathbf{M}\)-matrices plays an important role in matrix analysis, nonnegative matrix theory, numerical linear algebra, and the theory of iterative methods; see, for example, \cite{berman1994nonnegative,Varga2000Matrix,Young1971iterative}. A real matrix \(A\) is an \(\mathbf{M}\)-matrix if it can be written as
\[
A=sI-B,\qquad B\geq0,\quad s\geq\rho(B),
\]
and it is invertible when \(s>\rho(B)\). Here, $\rho(A)$ denotes the spectral radius of $A$, defined as the maximum of the absolute values of the eigenvalues of $A$, and $A \geq 0$ means that all the entries of $A$ are nonnegative.

One of the most important and most relevant properties is that \(A^{-1}\geq0\). In fact, there are almost 50 characterizations of the class of invertible $\mathbf{M}$-matrices \cite{berman1994nonnegative,cps}.  Although the class of \(\mathbf{M}\)-matrices is not closed under the Hadamard product in general, an important special case occurs when an \(\mathbf{M}\)-matrix is combined with the inverse of an \(\mathbf{M}\)-matrix. In particular, if \(A\) is an invertible \(\mathbf{M}\)-matrix, then \(A\circ A^{-1}\) is again an \(\mathbf{M}\)-matrix \cite{johnson1977ahadamard}. This naturally leads to the study of the spectral properties of \(A\circ A^{-1}\).

For a square matrix \(C\), define
$
q(C)=\min\{|\lambda|:\lambda\in\sigma(C)\}.
$
Thus, \(q(C)\) denotes the minimum modulus of the eigenvalues of \(C\). In particular, when \(C\) is an invertible \(\mathbf{M}\)-matrix, by the Perron-Frobenius theorem, \(q(C)\) is a positive eigenvalue of \(C\) and satisfies
$
q(C)=\frac{1}{\rho(C^{-1})}.
$
One of the motivations to study $q(A\circ A^{-1})$ comes from the problem of positive diagonal symmetrization of real matrices \cite{fiedler1985atrace,johnson1977ahadamard}. In this setting, the spectral properties of $A\circ A^{-1}$ provide information about whether $A$ can be transformed into a symmetric matrix by positive diagonal scaling. Further studies on $q(A\circ A^{-1})$ developed by discussing their bounds when $A$ belongs to a particular class of matrices. Fiedler, Johnson, Markham and Neumann \cite{fiedler1985atrace} proved that
$
0<q(A\circ A^{-1})\leq 1
$, when $A$ is an invertible $\mathbf{M}$-matrix.
For an irreducible $\mathbf{M}$-matrix $A$, they further showed that equality in the upper bound holds if and only if $A$ is positive diagonally symmetrizable. Subsequently, Fiedler and Markham \cite{fiedler1988aninequality} established the lower bound \(\frac{1}{n}\) and conjectured the stronger inequality
$
q(A\circ A^{-1})\geq\frac{2}{n}.
$
This conjecture was later proved independently in \cite{chen2004alower,song2000onaninequality,Yong2000proof}. Since then, considerable attention has been devoted to obtaining refinements of this inequality. Bounds involving individual entries of \(A\) and \(A^{-1}\), as well as refinements for various subclasses of \(\mathbf{M}\)-matrices, have also been proposed.

The preceding refinements are generally matrix-dependent or apply to particular subclasses of $\mathbf{M}$-matrices. It is therefore useful to distinguish such refinements from a uniform improvement of the bound $\frac{2}{n}$ over the full class of invertible $\mathbf{M}$-matrices. To complete this aspect of the discussion, we discuss an example for which
\[
\inf\left\{
q(A\circ A^{-1}):
A\text{ is an invertible }n\times n\ \mathbf{M}\text{-matrix}
\right\}
=\frac{2}{n}.
\]

\noindent Hence, the lower bound $\frac{2}{n}$ is sharp in the limiting sense for the full class of invertible $\mathbf{M}$-matrices. Entry-dependent estimates may, however, provide stronger bounds for particular matrices or restricted subclasses.

The extremal behavior associated with this lower bound is also closely connected with the Jacobi iterative method. If
$
A=D-L-U
$
is the usual splitting of an invertible \(\mathbf{M}\)-matrix, then
$
J_A=D^{-1}(L+U)
$
denotes its Jacobi iteration matrix. Xiang \cite{xiang2003hadamard} obtained lower bounds for \(q(A\circ A^{-1})\) in terms of \(\rho(J_A)\) and, in particular, showed that for a sequence \((A_k)\) of \(n\times n\) invertible \(\mathbf{M}\)-matrices, with \(n>2\),
$
q(A_k\circ A_k^{-1})\to\frac{2}{n}$ implies that $\rho(J_{A_k})\to1.
$
This naturally raises the question of whether the converse is true. We show that \(\rho(J_{A_k})\to1\) does not, in general, imply \(q(A_k\circ A_k^{-1})\to \frac{2}{n}\). We then identify a class of matrices for which the converse does hold, thereby obtaining a class of matrices for which the optimal lower bound is approached.

Another major objective of this paper is to investigate analogous questions for the broader class of real \(\mathbf{H}\)-matrices. For \(A=(a_{ij})\), its comparison matrix is defined by
\[
\mathcal{M}(A)=(m_{ij}),\qquad
m_{ii}=|a_{ii}|,\qquad
m_{ij}=-|a_{ij}|\quad(i\neq j).
\]
A matrix \(A\) is called an \(\mathbf{H}\)-matrix if \(\mathcal{M}(A)\) is an \(\mathbf{M}\)-matrix. Thus, the class of \(\mathbf{H}\)-matrices extends the class of \(\mathbf{M}\)-matrices while retaining several useful analytical and numerical properties. These matrices arise naturally in the analysis of iterative methods for systems of linear equations; see, for example, \cite{liu2006algorithmic,huang2007jacobi}. They are also connected with linear complementarity problems and other problems in numerical linear algebra and applied mathematics; see, for example, \cite{alanelli2007iterative,bru2008generalH,fiedler1967diagonal,horn1991matrix,Varga2000Matrix,Young1971iterative}.

In view of the close relationship between \(\mathbf{M}\)-matrices and \(\mathbf{H}\)-matrices, it is natural to ask whether lower bounds analogous to those for \(q(A\circ A^{-1})\) in the \(\mathbf{M}\)-matrix setting remain valid for invertible \(\mathbf{H}\)-matrices, that is those matrices for which $\mathcal{M}(A)$ are invertible $\mathbf{M}$-matrices. The class of invertible $\mathbf{H}$-matrices is denoted by $\mathcal{H}_I$ \cite{bru2008generalH,bru2009schurH}.We show that the situation for this class is substantially different. We show that \(A\circ A^{-1}\) is an \(\mathbf{H}\)-matrix whenever \(A\) is an invertible \(\mathbf{H}\)-matrix. But however there is no positive uniform lower bound, depending only on \(n\), for \(q(A\circ A^{-1})\) over the entire class. 

The comparison matrix provides another natural connection with the \(\mathbf{M}\)-matrix setting. For an invertible \(\mathbf{H}\)-matrix \(A\), we study
$
A\circ\mathcal{M}(A)^{-1}
$
and establish lower bounds for \(q(A\circ\mathcal{M}(A)^{-1})\) analogous to those known for \(q(A\circ A^{-1})\) when \(A\) is an invertible \(\mathbf{M}\)-matrix. We also investigate corresponding estimates involving the Jacobi iteration matrix. These results clarify which lower-bound properties of the \(\mathbf{M}\)-matrix case remain valid, and which fail, when the class is enlarged to \(\mathbf{H}\)-matrices. Finally, we extend the classical upper bound $q(A\circ A^{-1})\leq 1$ from invertible $\mathbf{M}$-matrices to $2\times 2$ invertible $\mathbf{H}$-matrices and to positive diagonally symmetrizable (PDS) invertible $\mathbf{H}$-matrices, with a complete characterization of equality in the irreducible case.

The major contributions of this paper are as follows:
\begin{enumerate}
\item We discuss the sharpness of the lower bound $\frac{2}{n}$ for $q(A\circ A^{-1})$, when $A$ is an invertible $\mathbf{M}$-matrix (Section \ref{m-matrix bound}). We then show that the converse of the known implication for $\mathbf{M}$-matrices $A_k$,
$q(A_k\circ A_k^{-1})\to\frac{2}{n} \Rightarrow \rho(J_{A_k})\to1$
fails in general (Example \ref{conv not true eg}), and identify a class for which the converse holds (Section \ref{optimal matrices}).
\item We discuss some structural properties of $\mathbf{H}$-matrices (Section \ref{h mat results}) and prove that the Hadamard product of an $\mathbf{H}$-matrix and the inverse of an invertible $\mathbf{H}$-matrix is again an $\mathbf{H}$-matrix (Corollary \ref{a circ ainv is h}).
\item For invertible $\mathbf{H}$-matrices, we show that there is no positive dimension-dependent lower bound for $q(A\circ A^{-1})$ in general (Example \ref{ex:Hmatrix_counterexample}), and establish lower bounds for $q(A\circ\mathcal{M}(A)^{-1})$ (Theorem \ref{thm:Hadamard_comparison_inverse}), including a direct consequence involving the Jacobi iteration matrix (Theorem \ref{h matrix jacobi bound}). Similar to the class of $\mathbf{M}$-matrices, we give a sufficient condition for the Jacobi iterative matrix $J_{\mathcal{M}(A)}$ to tend to 1 (Proposition \ref{prop:H-q-zero-jacobi}).

\item We investigate the upper bound for $q(A\circ A^{-1})$ when $A$ is an invertible $\mathbf{H}$-matrix. We prove that $q(A\circ A^{-1})\leq 1$ for $2\times 2$ invertible $\mathbf{H}$-matrices (Proposition~\ref{prop:h-2by2-upper}) and for positive diagonally symmetrizable invertible $\mathbf{H}$-matrices (Theorem~\ref{thm:PDS-H-upper}). We also characterize the equality case for irreducible PDS invertible $\mathbf{H}$-matrices (Theorem~\ref{thm:PDS-H-equality-characterization}).
\end{enumerate}

\section{Hadamard Products Involving $\mathbf{M}$-matrices}
\label{m-matrices section}

In this section, we focus on the classical lower bound $\frac{2}{n}$ for $q(A\circ A^{-1})$, where $A$ is an invertible $\mathbf{M}$-matrix of order $n$. We first revisit the sharpness of this bound and then investigate its connection with the Jacobi iteration matrix, with particular attention to the converse of a known limiting implication.

Several refinements of the lower bound $\frac{2}{n}$ have been obtained by incorporating additional information about the entries or structure of the matrix, particularly since $\frac{2}{n}$ becomes small as $n$ increases. For instance, Li et al.~\cite{Li2007} derived entry-dependent lower bounds for invertible $\mathbf{M}$-matrices. The sharpness of $\frac{2}{n}$ in the limiting sense was already discussed in \cite{wrong proof}. However, as the proof of the Fiedler--Markham conjecture presented there was later found to be incorrect, this sharpness observation appears to have been overlooked in the subsequent literature. We revisit this observation to acknowledge the earlier contribution and complete the discussion of the limiting sharpness of $\frac{2}{n}$.

\subsection*{Sharpness of the lower bound}
\label{m-matrix bound}

Fix $n\geq3$, and let $P\in\mathbb{R}^{n\times n}$ be the cyclic permutation matrix defined by $P(e_i)=e_{i-1}$ for $i=2,\ldots,n$ and $P(e_1)=e_n$. For $t\in(0,1)$, set $A_t=I-tP$. Since $P\geq0$ and $\rho(P)=1$, we have $\rho(tP)=t<1$, so $A_t$ is an invertible $\mathbf{M}$-matrix. Moreover, since $P^n=I$, we have, $(I - tP)\Bigl(\sum_{k=0}^{n-1} t^k P^k\Bigr) = I - t^n I.$ Thus 
\[
A_t^{-1}
=
\frac{1}{1-t^n}\sum_{k=0}^{n-1}t^kP^k,
\qquad
A_t\circ A_t^{-1}
=
\frac{1}{1-t^n}(I-t^2P).
\]
Let $\omega=e^{2\pi i/n}$. Since the eigenvalues of $P$ are $\omega^j$, $j=0,\ldots,n-1$, we obtain
\[
\sigma(A_t\circ A_t^{-1})
=
\left\{
\frac{1-t^2\omega^j}{1-t^n}:j=0,\ldots,n-1
\right\}.
\]
By the reverse triangle inequality,
$
|1-t^2\omega^j|\geq 1-t^2,
$
with equality for $j=0$. Hence
\[
q(A_t\circ A_t^{-1})
=
\frac{1-t^2}{1-t^n}
=
\frac{1+t}{1+t+\cdots+t^{n-1}}
\to \frac{2}{n}
\qquad (t\to 1^-).
\]
Therefore, for every $\varepsilon>0$, there exists $t\in(0,1)$ such that
$
q(A_t\circ A_t^{-1})
<
\frac{2}{n}+\varepsilon.
$
Combining this with the known lower bound
$
q(A\circ A^{-1})>\frac{2}{n}
$
for invertible $\mathbf{M}$-matrices of order $n$, we obtain
\[
\inf\left\{
q(A\circ A^{-1}):
A\text{ is an invertible }n\times n\ \mathbf{M}\text{-matrix}
\right\}
=
\frac{2}{n}.
\]
Hence, the lower bound $\frac{2}{n}$ is sharp in the limiting sense.

\subsection{Bounds Involving the Jacobi Iteration Matrix}\label{jacobi m}

Let $A=D-L-U$, where $D$ is the diagonal part of $A$, and $L$ and $U$ are its strictly lower and strictly upper triangular parts, respectively. The Jacobi iteration matrix associated with $A$ is
$
J_A=D^{-1}(L+U).
$ The convergence of the Jacobi method is determined by $\rho(J_A)$: the method converges when $\rho(J_A)<1$, and its convergence becomes increasingly slow as $\rho(J_A)$ approaches $1$.

For a sequence ${A_k}$ of $n\times n$ invertible $\mathbf{M}$-matrices, it was shown in \cite{xiang2003hadamard} that
$
q(A_k\circ A_k^{-1})\to\frac{2}{n}
$
implies
$
\rho(J_{A_k})\to1.
$ Thus, $\rho(J_{A_k})\to1$ is necessary for $q(A_k\circ A_k^{-1})$ to approach its infimum $\frac{2}{n}$. This naturally raises the question of whether the converse is also true:
$
\rho(J_{A_k})\to1
\quad\Longrightarrow\quad
q(A_k\circ A_k^{-1})\to\frac{2}{n}\,?
$
The following bound from \cite{xiang2003hadamard} is relevant to this question:

$$
q(A\circ A^{-1})\geq
\max\left\{
1-\rho(J_A)^2,\,
\frac{1+\rho(J_A)^{1/(n+2)}}
{1+(n-1)\rho(J_A)^{1/(n+2)}}
\right\}.
$$

As $\rho(J_A)\to1$, the second term on the right-hand side approaches $\frac{2}{n}$. This, however, provides only a lower bound and therefore does not establish the converse implication. We show that the converse is in fact false by constructing a simple counterexample. We then identify a class of invertible $\mathbf{M}$-matrices for which the converse does hold.

\begin{ex}[A Counterexample to the Converse Implication]\label{conv not true eg}
    
{\rm
Consider $n=3$ and let
$$
A_\varepsilon=
\begin{pmatrix}
1 & -\varepsilon & -\varepsilon\\
-\varepsilon & 1 & -\varepsilon\\
-\varepsilon & -\varepsilon & 1
\end{pmatrix},
\qquad 0<\varepsilon<\frac12.
$$

Since $A_\varepsilon=I-B_\varepsilon$, where $B_\varepsilon\geq0$ and $\rho(B_\varepsilon)=2\varepsilon<1$, the matrix $A_\varepsilon$ is an invertible $\mathbf{M}$-matrix. Moreover, since the diagonal part of $A_\varepsilon$ is $I$, its Jacobi iteration matrix is
\[
J_{A_\varepsilon}=
\begin{pmatrix}
0 & \varepsilon & \varepsilon\\
\varepsilon & 0 & \varepsilon\\
\varepsilon & \varepsilon & 0
\end{pmatrix}.
\]
The eigenvalues of $J_{A_\varepsilon}$ are $2\varepsilon$ and $-\varepsilon$, with $-\varepsilon$ of multiplicity two. Hence, $\rho(J_{A_\varepsilon})=2\varepsilon\to1$ as $\varepsilon\uparrow\frac12$.
On the other hand, writing
$A_\varepsilon=(1+\varepsilon)I-\varepsilon\mathbf{1}\mathbf{1}^{\top}$,
the Sherman--Morrison formula gives
$$
A_\varepsilon^{-1}
=\frac{1}{(1+\varepsilon)(1-2\varepsilon)}
\begin{pmatrix}
1-\varepsilon & \varepsilon & \varepsilon\\
\varepsilon & 1-\varepsilon & \varepsilon\\
\varepsilon & \varepsilon & 1-\varepsilon
\end{pmatrix}.
$$

Therefore,
$$
A_\varepsilon\circ A_\varepsilon^{-1}
=\frac{1}{(1+\varepsilon)(1-2\varepsilon)}
\begin{pmatrix}
1-\varepsilon & -\varepsilon^2 & -\varepsilon^2\\
-\varepsilon^2 & 1-\varepsilon & -\varepsilon^2\\
-\varepsilon^2 & -\varepsilon^2 & 1-\varepsilon
\end{pmatrix}.
$$

Its eigenvalues are
$
\frac{1-\varepsilon-2\varepsilon^2}
{(1+\varepsilon)(1-2\varepsilon)}$ and 
$\frac{1-\varepsilon+\varepsilon^2}
{(1+\varepsilon)(1-2\varepsilon)},
$ where the second eigenvalue has multiplicity two. Since
$1-\varepsilon-2\varepsilon^2=(1+\varepsilon)(1-2\varepsilon)$,
the first eigenvalue is $1$. Moreover, the second eigenvalue is greater than $1$ for $0<\varepsilon<\frac12$. Consequently, $q(A_\varepsilon\circ A_\varepsilon^{-1})=1$ for all $0<\varepsilon<\frac12$.
Thus, $\rho(J_{A_\varepsilon})\to1$, but $q(A_\varepsilon\circ A_\varepsilon^{-1})=1\not\to\frac23$.
Hence $\rho(J_{A_k})\to1$ is not sufficient for
$q(A_k\circ A_k^{-1})\to\frac{2}{n}$. The converse implication therefore fails even for symmetric invertible $\mathbf{M}$-matrices.}
\end{ex}
\subsubsection{A Class for which the converse holds}
\label{optimal matrices}

Although the converse fails in general, the cyclic family introduced in Section \ref{m-matrix bound} provides a natural class for which it holds. Fix $n\geq3$, and let $P\in\mathbb{R}^{n\times n}$ be the cyclic permutation matrix defined by $P(e_i)=e_{i-1}$ for $i=2,\ldots,n$ and $P(e_1)=e_n$. For $0<t<1$, let $A_t=I-tP$. As shown in Section \ref{m-matrix bound}, $A_t$ is an invertible $\mathbf{M}$-matrix and $q(A_t\circ A_t^{-1})\to\frac{2}{n}$ as $t\to1^-$. Since the diagonal part of $A_t$ is $I$, its Jacobi iteration matrix is $J_{A_t}=tP$. As $\rho(P)=1$, we also have $\rho(J_{A_t})=t\to1$. Thus, for this cyclic family, $\rho(J_{A_t})\to1$ and $q(A_t\circ A_t^{-1})\to\frac{2}{n}$.

This construction extends naturally to permutation matrices with several cycles. In fact, the cycle structure determines the limiting value completely, as shown in the following theorem.

\begin{thm}\label{thm:permutation-cycles}
{\rm Let $P\in\mathbb{R}^{n\times n}$ be a permutation matrix without fixed points, with cycle lengths $\ell_1,\ldots,\ell_m\geq2$, and set $L=\max_i\ell_i$. For $0<t<1$, define $A_t=I-tP$. Then $A_t$ is an invertible $\mathbf{M}$-matrix, $J_{A_t}=tP$, and $\rho(J_{A_t})=t$. Moreover, $q(A_t\circ A_t^{-1})=\frac{1-t^2}{1-t^L}$, $\quad \lim_{t\to1^-}q(A_t\circ A_t^{-1})=\frac{2}{L}$. In particular, $\lim_{t\to1^-}q(A_t\circ A_t^{-1})=\frac{2}{n}$ if and only if $P$ consists of a single cycle of length $n$.}
\end{thm}

\begin{proof}
Since $P\geq0$ and $\rho(tP)=t<1$, the Neumann series
$$
A_t^{-1}=(I-tP)^{-1}=\sum_{k=0}^{\infty}(tP)^k
$$
converges and is nonnegative. Since $A_t$ has nonpositive off-diagonal entries, it follows that $A_t$ is an invertible $\mathbf{M}$-matrix. Moreover, since $P$ has no fixed points, the diagonal part of $A_t$ is $I$, and hence $J_{A_t}=tP$ and $\rho(J_{A_t})=t$. Up to permutation similarity,
$$
P\sim\operatorname{diag}(P_{\ell_1},\ldots,P_{\ell_m}), \qquad A_t\sim\operatorname{diag}(I-tP_{\ell_1},\ldots,I-tP_{\ell_m}),
$$
where $P_{\ell_i}$ denotes a cyclic permutation matrix of order $\ell_i$. If $Q$ is the corresponding permutation matrix, then
$
Q^T(A_t\circ A_t^{-1})Q=(Q^TA_tQ)\circ(Q^TA_t^{-1}Q),
$ so $A_t\circ A_t^{-1}$ is permutation-similar to the corresponding block diagonal matrix. We may therefore compute $q(A_t\circ A_t^{-1})$ blockwise.

For a cycle of length $\ell\geq2$, we have $P_\ell^\ell=I$, and hence

$$
(I-tP_\ell)^{-1}=\frac{1}{1-t^\ell}\sum_{k=0}^{\ell-1}t^kP_\ell^k.
$$

It follows that

$$
(I-tP_\ell)\circ(I-tP_\ell)^{-1}=\frac{1}{1-t^\ell}(I-t^2P_\ell).
$$

Since the eigenvalues of $P_\ell$ are the $\ell$th roots of unity $\omega_\ell^j$, $j=0,\ldots,\ell-1$, we obtain

$$
q\bigl((I-tP_\ell)\circ(I-tP_\ell)^{-1}\bigr)=\min_{0\leq j<\ell}\frac{|1-t^2\omega_\ell^j|}{1-t^\ell}=\frac{1-t^2}{1-t^\ell},
$$

where the minimum occurs at $\omega_\ell^0=1$.

Combining the blocks gives

$$
q(A_t\circ A_t^{-1})=\min_{1\leq i\leq m}\frac{1-t^2}{1-t^{\ell_i}}=\frac{1-t^2}{1-t^L},
$$

because $1-t^\ell$ increases with $\ell$ for $0<t<1$. Therefore,

$$
\lim_{t\to1^-}q(A_t\circ A_t^{-1})=\frac{2}{L}.
$$

\noindent The limit equals $\frac{2}{n}$ if and only if $L=n$, which is equivalent to $P$ consisting of a single cycle of length $n$.
\end{proof}

\begin{rem}
{\rm The assumption that $P$ has no fixed points in Theorem~\ref{thm:permutation-cycles} is needed for the identity $J_{A_t}=tP$. If $P=I$, then $A_t=(1-t)I$, so $J_{A_t}=0$ and $q(A_t\circ A_t^{-1})=1$. If $P$ has both fixed points and nontrivial cycles of lengths $\ell_1,\ldots,\ell_m$, with $L=\max_i\ell_i$, then the fixed points contribute $1\times1$ blocks $[1]$, while the nontrivial cycles give

$$
q(A_t\circ A_t^{-1})
=\min\left\{1,\frac{1-t^2}{1-t^L}\right\}
=\frac{1-t^2}{1-t^L}.
$$

Moreover, the Jacobi iteration matrix has zero blocks corresponding to the fixed points and blocks $tP_{\ell_i}$ corresponding to the nontrivial cycles. Hence $\rho(J_{A_t})=t\to1$ whenever $P$ has at least one nontrivial cycle.

Thus, within the family $A_t=I-tP$, the condition $\rho(J_{A_t})\to1$ alone does not determine the limiting behavior of $q(A_t\circ A_t^{-1})$. If $P$ has at least one nontrivial cycle, then

$$
\lim_{t\to1^-}q(A_t\circ A_t^{-1})=\frac{2}{L},
$$
where $L$ is the length of its largest nontrivial cycle. Consequently, $\rho(J_{A_t})\to1$ leads to the extremal limit $q(A_t\circ A_t^{-1})\to\frac{2}{n}$ precisely when $P$ consists of a single cycle of length $n$.}
\end{rem}

\section{Hadamard Products Involving $\mathbf{H}$-Matrices}
\label{h-matrix section}
The aim of this section is to study Hadamard products involving invertible $\mathbf{H}$-matrices, with particular emphasis on bounds for the minimum modulus of their eigenvalues. We first establish some structural properties of $\mathbf{H}$-matrices and their inverses that will be useful in our analysis. We then examine $A\circ A^{-1}$ and show that, unlike the $\mathbf{M}$-matrix case, no positive lower bound exists, in general. This motivates the consideration of $A\circ\mathcal{M}(A)^{-1}$, for which positive lower bounds can be obtained. We conclude with an upper bound for $q(A\circ A^{-1})$ for symmetric $\mathbf{H}$-matrices.

We shall use the following well-known characterization of invertible $\mathbf{H}$-matrices. A matrix $A=(a_{ij})$ is an invertible $\mathbf{H}$-matrix if and only if there exists a positive vector $d=(d_i)$ such that

$$
|a_{ii}|d_i>
\sum_{\substack{j=1\\j\neq i}}^n |a_{ij}|d_j,
\qquad i=1,\ldots,n.
$$

Equivalently, there exists a positive diagonal matrix
$D=\operatorname{diag}(d_1,\ldots,d_n)$ such that $AD$ is strictly row diagonally dominant \cite{bru2008generalH,bru2009schurH}. 

\subsection{Structural properties}
\label{h mat results}

We begin with some properties of diagonally dominant matrices and their inverses that will be useful in the subsequent analysis. The following result of Fiedler and Markham \cite{fiedler1988aninequality} provides the starting point.

\begin{pro}[Proposition 2, \cite{fiedler1988aninequality}]
{\rm Let $C$ be a diagonally dominant $\mathbf{M}$-matrix, and write $C^{-1}=(\gamma_{ij})$. Then $\gamma_{ii}>\gamma_{ji}$ for all $i\neq j$.}
\end{pro}

We next extend this property to general invertible row diagonally dominant matrices.

\begin{pro}\label{lemma:RSDD-inverse-column}
{\rm Let $A=(a_{ij})$ be an invertible row diagonally dominant matrix, and let $A^{-1}=(\alpha_{ij})$. Then
\begin{equation}\label{equation:RSDD1}
|\alpha_{ii}|\geq|\alpha_{ji}|,\qquad i\neq j.
\end{equation}
Moreover, if $A$ is strictly row diagonally dominant, then the inequalities in \eqref{equation:RSDD1} are strict.}
\end{pro}

\begin{proof}
We first consider the strictly row diagonally dominant case. Fix $r\in\{1,\ldots,n\}$ and set $\beta_r=\displaystyle\max_{1\leq k\leq n}|\alpha_{kr}|$. For $i\neq r$, the $(i,r)$ entry of $AA^{-1}=I$ gives
$
a_{ii}\alpha_{ir}=-\sum_{j\neq i}a_{ij}\alpha_{jr}.
$
Hence,
\[
|a_{ii}|\,|\alpha_{ir}|
\leq\sum_{j\neq i}|a_{ij}|\,|\alpha_{jr}|
\leq\beta_r\sum_{j\neq i}|a_{ij}|
<\beta_r|a_{ii}|.
\]
Since $a_{ii}\neq0$, we obtain $|\alpha_{ir}|<\beta_r$ for every $i\neq r$. Thus, $\beta_r$ can only be attained at $\alpha_{rr}$, and consequently $|\alpha_{ir}|<|\alpha_{rr}|$ for all $i\neq r$.

Now suppose that $A$ is invertible and row diagonally dominant. Notice that $a_{ii}\neq0$ for every $i$; otherwise, row diagonal dominance would force the entire $i$th row of $A$ to be zero, contradicting invertibility. Let
$
S=\operatorname{diag}\bigl(\operatorname{sgn}(a_{11}),\ldots,\operatorname{sgn}(a_{nn})\bigr)
$
and, for $\varepsilon>0$, set $A_\varepsilon=A+\varepsilon S$. Then $A_\varepsilon$ is strictly row diagonally dominant, since
$
|(A_\varepsilon)_{ii}|=|a_{ii}|+\varepsilon>\sum_{j\neq i}|a_{ij}|.
$
By the first part, $|(A_\varepsilon^{-1})_{ir}|<|(A_\varepsilon^{-1})_{rr}|$ for all $i\neq r$. Since $A_\varepsilon\to A$ as $\varepsilon\to0^+$ and $A$ is invertible, $A_\varepsilon^{-1}\to A^{-1}$. Therefore,
$
(A_\varepsilon^{-1})_{ir}\to\alpha_{ir},$
$(A_\varepsilon^{-1})_{rr}\to\alpha_{rr}.
$
Passing to the limit gives $|\alpha_{ir}|\leq|\alpha_{rr}|$ for all $i\neq r$. Since $r$ was arbitrary, \eqref{equation:RSDD1} follows.
\end{proof}

\begin{rem}
{\rm \begin{enumerate}[(a)]
\item An analogous statement holds for column diagonally dominant matrices, with the roles of rows and columns interchanged.

\item The converse of Proposition~\ref{lemma:RSDD-inverse-column} holds for matrices of order $2$, but fails in general for matrices of higher order. In particular, the following example gives an invertible $3\times3$ matrix that is not row diagonally dominant, although its inverse satisfies \eqref{equation:RSDD1}.
\end{enumerate}}
\end{rem}

\begin{ex}\label{eg1}
{\rm Let
$
A=
\begin{pmatrix}
\ 14 & 9 & -8\\
-8 & 18 & -7\\
\ \ 9 & 0 & \ 18
\end{pmatrix}.
$
Clearly, $A$ is not row diagonally dominant. However,
$
A^{-1}
=
\frac{1}{81}
\begin{pmatrix}
4 & -2 & 1\\
1 & 4 & 2\\
-2 & 1 & 4
\end{pmatrix},
$
and $A^{-1}$ satisfies the inequalities in \eqref{equation:RSDD1}. Thus, the converse of Proposition~\ref{lemma:RSDD-inverse-column} does not hold in general.}
\end{ex}

The characterization of invertible $\mathbf{H}$-matrices immediately implies that, for an invertible $\mathbf{H}$-matrix $A=(a_{ij})$, there exists a positive vector $d>0$ such that
$
|a_{ii}|d_i>|a_{ij}|d_j
\quad \text{for all } i\neq j.
$
The following result shows that an analogous property holds for the inverse and will be useful in establishing Theorem~\ref{a hadamard b inv is h} in the next subsection.

\begin{pro}\label{theorem:inverse_H_row_S_Diagonal_dominant}
{\rm Let $A\in\mathbb{R}^{n\times n}$ be an invertible $\mathbf{H}$-matrix with
$A^{-1}=(\alpha_{ij})$. Then there exists a positive vector $\tilde d>0$ such that
$
|\alpha_{ii}|\tilde d_i>|\alpha_{ij}|\tilde d_j
$ {for all } $i\neq j.
$}
\end{pro}

\begin{proof}
Since $A^T$ is an invertible $\mathbf{H}$-matrix, there exists a positive diagonal matrix
$E=\operatorname{diag}(e_{11},\ldots,e_{nn})$ such that $A^TE$ is strictly row diagonally dominant. Equivalently, $EA$ is strictly column diagonally dominant. By the column analogue of Proposition~\ref{lemma:RSDD-inverse-column}, and since $(EA)^{-1}=A^{-1}E^{-1}$, we have
\[
|\alpha_{ii}e_{ii}^{-1}|>|\alpha_{ij}e_{jj}^{-1}|,
\qquad i\neq j.
\]
Setting $\tilde d_i=e_{ii}^{-1}$ gives
$|\alpha_{ii}|\tilde d_i>|\alpha_{ij}|\tilde d_j$ for all $i\neq j$, as required.
\end{proof}

\begin{rem}\label{h matrix remark}
{\rm The converse of Proposition~\ref{theorem:inverse_H_row_S_Diagonal_dominant} does not hold in general. Indeed, let
\[
A=
\begin{pmatrix}
\frac{5}{3} & \frac{4}{3} & 1\\
\frac{4}{3} & \frac{5}{3} & 1\\
1 & 1 & 1
\end{pmatrix},
\qquad
A^{-1}=
\begin{pmatrix}
\ \ 2 & -1 & -1\\
-1 & \ \ 2 & -1\\
-1 & -1 & \ \ 3
\end{pmatrix}.
\]
Then $A^{-1}=(\alpha_{ij})$ satisfies
$|\alpha_{ii}|\tilde d_i>|\alpha_{ij}|\tilde d_j$ for all $i\neq j$, with
$\tilde d=(1,1,1)^{\mathsf T}$. However, $A$ is not an $\mathbf{H}$-matrix, since its comparison matrix is not an $\mathbf{M}$-matrix.}
\end{rem}

\subsection{Hadamard products and spectral bounds}
\label{h hadamard bounds}

In this subsection, we study Hadamard products involving invertible $\mathbf{H}$-matrices, with particular emphasis on their spectral properties. We first establish that $A\circ A^{-1}$ preserves the $\mathbf{H}$-matrix structure. We then investigate lower bounds for $q(A\circ A^{-1})$ and $q(A\circ\mathcal{M}(A)^{-1})$, including bounds involving the Jacobi iteration matrix. Finally, we consider upper bounds for $q(A\circ A^{-1})$ and formulate a related open problem.

We begin by recalling the following well-known closure property of $\mathbf{H}$-matrices under the Hadamard product.

\begin{thm}[\cite{lynn1964}]
{\rm Let $A,B\in\mathbb{R}^{n\times n}$ be $\mathbf{H}$-matrices. Then $A\circ B$ is an $\mathbf{H}$-matrix.}
\end{thm}

The preceding result cannot be applied directly to $A\circ A^{-1}$, since $A^{-1}$ need not be an $\mathbf{H}$-matrix. Nevertheless, the structural properties established in the previous subsection allow us to obtain the desired conclusion. In fact, we prove the following more general result. 
\begin{thm}\label{a hadamard b inv is h} {\rm Let $A\in\mathbb{R}^{n\times n}$ be an $\mathbf{H}$-matrix and let $B\in\mathbb{R}^{n\times n}$ be an invertible $\mathbf{H}$-matrix. Then $A\circ B^{-1}$ is an $\mathbf{H}$-matrix.} 
\end{thm}
\begin{proof} Let $A=(a_{ij})$ and $B^{-1}=(\beta_{ij})$. Since $A$ is an $\mathbf{H}$-matrix, there exists a positive vector $d=(d_i)$ such that \begin{equation}\label{defn h} |a_{ii}|d_i> \sum_{\substack{j=1\\j\neq i}}^n |a_{ij}|d_j, \qquad i=1,\ldots,n. \end{equation} Since $B$ is an invertible $\mathbf{H}$-matrix, Proposition~\ref{theorem:inverse_H_row_S_Diagonal_dominant} ensures that there exists a positive vector $\tilde d=(\tilde d_i)$ such that \begin{equation}\label{eq:H_AoinverseB_beta} |\beta_{ii}|\tilde d_i> |\beta_{ij}|\tilde d_j, \qquad i\neq j. \end{equation} Define $\hat d=(\hat d_i)>0$ by $\hat d_i=d_i\tilde d_i$. Then, for each $i=1,\ldots,n$, \begin{align*} |a_{ii}\beta_{ii}|\hat d_i &=|a_{ii}|d_i\,|\beta_{ii}|\tilde d_i\\ &> \left(\sum_{\substack{j=1\\j\neq i}}^n |a_{ij}|d_j\right) |\beta_{ii}|\tilde d_i \qquad\text{(by~\eqref{defn h})}\\ &> \sum_{\substack{j=1\\j\neq i}}^n |a_{ij}|d_j\,|\beta_{ij}|\tilde d_j \qquad\text{(by~\eqref{eq:H_AoinverseB_beta})}\\ &= \sum_{\substack{j=1\\j\neq i}}^n |a_{ij}\beta_{ij}|\hat d_j. \end{align*} Hence, $A\circ B^{-1}$ is generalized strictly row diagonally dominant and therefore is an $\mathbf{H}$-matrix. 
\end{proof} 

Taking $B=A$ in Theorem~\ref{a hadamard b inv is h} immediately yields the following consequence.
\begin{cor}\label{a circ ainv is h} {\rm Let $A$ be an invertible $\mathbf{H}$-matrix. Then $ A\circ A^{-1}$ is an $\mathbf{H}$-matrix. }
\end{cor}

\subsubsection{Lower spectral bounds}
\label{subsubsec:h-lower-bounds}
Having shown that $A\circ A^{-1}$ is an $\mathbf{H}$-matrix whenever $A$ is an invertible $\mathbf{H}$-matrix, it is natural to ask whether a positive lower bound depending only on the order of the matrix remains valid when the class of invertible $\mathbf{M}$-matrices is enlarged to invertible $\mathbf{H}$-matrices. More precisely, one may ask whether there exists a constant $c_n>0$, depending only on $n$, such that
\[
q(A\circ A^{-1})\geq c_n
\]
for every invertible $\mathbf{H}$-matrix $A$ of order $n$. The following example shows that the answer is negative: even in dimension two, $q(A\circ A^{-1})$ can be made arbitrarily small.

\begin{ex}\label{ex:Hmatrix_counterexample}
{\rm Consider $A=\begin{pmatrix}1&-\alpha\\ \beta&1\end{pmatrix}$, where $\alpha,\beta>0$ and $\alpha\beta<1$. Its comparison matrix is $\mathcal{M}(A)=\begin{pmatrix}1&-\alpha\\ -\beta&1\end{pmatrix}$, with $\det\mathcal{M}(A)=1-\alpha\beta>0$. Thus, $\mathcal{M}(A)$ is an invertible $\mathbf{M}$-matrix, and hence $A$ is an invertible $\mathbf{H}$-matrix.

A direct computation gives $A^{-1}=\frac{1}{1+\alpha\beta}
\begin{pmatrix}1&\alpha\\ -\beta&1
\end{pmatrix}$ and
$A\circ A^{-1}=\frac{1}{1+\alpha\beta}\begin{pmatrix}1&-\alpha^2\\ -\beta^2&1\end{pmatrix}$.
The eigenvalues of $A\circ A^{-1}$ are $1$ and $\frac{1-\alpha\beta}{1+\alpha\beta}$. Therefore,
$q(A\circ A^{-1})=\frac{1-\alpha\beta}{1+\alpha\beta}\to0$ as $\alpha\beta\to1^-$. Hence, $q(A\circ A^{-1})$ has no positive lower bound over the class of invertible $\mathbf{H}$-matrices of order $2$.

The same conclusion holds for every $n\ge2$. Indeed, for $n>2$, taking $B=\operatorname{diag}(A,I_{n-2})\in\mathbb{R}^{n\times n}$ gives an invertible $\mathbf{H}$-matrix with $B\circ B^{-1}=\operatorname{diag}(A\circ A^{-1},I_{n-2})$. Thus,
$q(B\circ B^{-1})=q(A\circ A^{-1})\to0$ as $\alpha\beta\to1^-$. Hence, for every $n\ge2$, $q(A\circ A^{-1})$ can be made arbitrarily close to zero within the class of invertible $\mathbf{H}$-matrices.}
\end{ex}

The preceding example illustrates a key difference from the $\mathbf{M}$-matrix setting: an invertible $\mathbf{H}$-matrix need not be inverse-positive, so $A^{-1}$ may have mixed signs. This suggests replacing $A^{-1}$ by an entrywise nonnegative matrix that controls $|A^{-1}|$. By the classical result of Ostrowski \cite{ostro}, we have
$
|A^{-1}|\leq\mathcal{M}(A)^{-1}.
$
Moreover, since $\mathcal{M}(A)$ is an invertible $\mathbf{M}$-matrix, we have $\mathcal{M}(A)^{-1}\geq0$. 
Thus, $\mathcal{M}(A)^{-1}$ is a natural choice, leading to the following dimension-dependent lower bound for $q(A\circ\mathcal{M}(A)^{-1})$.

\begin{thm}\label{thm:Hadamard_comparison_inverse}
{\rm Let $A\in\mathbb{R}^{n\times n}$ be an invertible $\mathbf{H}$-matrix of order $n\geq2$. Then
\[
q\bigl(A\circ\mathcal{M}(A)^{-1}\bigr)\geq\frac{2}{n}.
\]}
\end{thm}

\begin{proof}
Let $M=\mathcal{M}(A)$ and $C=A\circ M^{-1}$. Since $M$ is an invertible $\mathbf{M}$-matrix, $M^{-1}\geq0$, and hence
\[
\mathcal{M}(C)
=\mathcal{M}(A\circ M^{-1})
=\mathcal{M}(A)\circ M^{-1}
=M\circ M^{-1}.
\]
As $M\circ M^{-1}$ is an invertible $\mathbf{M}$-matrix, $C$ is an invertible $\mathbf{H}$-matrix. Thus, Ostrowski's inequality gives
$
|C^{-1}|\leq\mathcal{M}(C)^{-1}.
$
Moreover, $\rho(X)\leq\rho(|X|)$ for any square matrix $X$, and the spectral radius is monotone for nonnegative matrices. Therefore,
\[
\rho(C^{-1})
\leq\rho\bigl(|C^{-1}|\bigr)
\leq\rho\bigl(\mathcal{M}(C)^{-1}\bigr),
\]
and hence
\[
q(C)
=\frac{1}{\rho(C^{-1})}
\geq\frac{1}{\rho\bigl(\mathcal{M}(C)^{-1}\bigr)}
=q\bigl(\mathcal{M}(C)\bigr).
\]
Using $\mathcal{M}(C)=M\circ M^{-1}$, we obtain
\begin{equation}\label{compareq*andq}
q\bigl(A\circ\mathcal{M}(A)^{-1}\bigr)
=q(C)
\geq q\bigl(M\circ M^{-1}\bigr).
\end{equation}
Finally, the known $\mathbf{M}$-matrix bound, together with \eqref{compareq*andq}, yields
$
q\bigl(A\circ\mathcal{M}(A)^{-1}\bigr)\geq q(M\circ M^{-1})\geq\frac{2}{n},
$ as required.
\end{proof}

\begin{rem}\label{rem:counterexample_resolved}
{\rm The behavior in Example~\ref{ex:Hmatrix_counterexample} is closely related to the fact that an invertible $\mathbf{H}$-matrix need not be inverse-positive. Replacing $A^{-1}$ by $\mathcal{M}(A)^{-1}$ gives a clearly different behavior. Indeed, for the same matrix
$
A=\begin{pmatrix}1&-\alpha\\ \beta&1\end{pmatrix},
$
where $\alpha,\beta>0$ and $\alpha\beta<1$, we have
$
\mathcal{M}(A)^{-1}
=\frac{1}{1-\alpha\beta}
\begin{pmatrix}1&\alpha\\ \beta&1\end{pmatrix}.
$
Hence, setting $C=A\circ\mathcal{M}(A)^{-1}$,
$
C=\frac{1}{1-\alpha\beta}
\begin{pmatrix}1&-\alpha^2\\ \beta^2&1\end{pmatrix}
$
and
$
C^{-1}
=\frac{1-\alpha\beta}{1+\alpha^2\beta^2}
\begin{pmatrix}1&\alpha^2\\ -\beta^2&1\end{pmatrix}.
$
Since the eigenvalues of the last matrix are $1\pm i\alpha\beta$, we obtain
\[
\rho(C^{-1})
=\frac{1-\alpha\beta}{\sqrt{1+\alpha^2\beta^2}},
\qquad
q(C)
=\frac{\sqrt{1+\alpha^2\beta^2}}{1-\alpha\beta}>1=\frac{2}{n}.
\]
Thus, in contrast to $q(A\circ A^{-1})$, which tends to zero as $\alpha\beta\to1^-$, we have
$
q(A\circ\mathcal{M}(A)^{-1})\to\infty.
$}
\end{rem}

We conclude this discussion with bounds involving the Jacobi iteration matrix. Xiang \cite[Theorem~3.1]{xiang2003hadamard} proved that, for an invertible $\mathbf{H}$-matrix $A$,
\begin{equation}\label{xiang bound}
q(A\circ A^{-1})
\geq
\frac{1-\rho(J_{\mathcal{M}(A)})^2}
     {1+\rho(J_{\mathcal{M}(A)})^2}.
\end{equation}
Although this gives a positive lower bound for each invertible $\mathbf{H}$-matrix, it is not bounded away from zero over the entire class, as the following example shows.

\begin{ex}
{\rm Consider the matrix $A$ from Example~\ref{ex:Hmatrix_counterexample}. In this case,
$\rho(J_{\mathcal{M}(A)})=\sqrt{\alpha\beta}$, and hence
\[
\frac{1-\rho(J_{\mathcal{M}(A)})^2}
     {1+\rho(J_{\mathcal{M}(A)})^2}
=
\frac{1-\alpha\beta}{1+\alpha\beta}
\to0
\qquad\text{as }\ \alpha\beta\to1^-.
\]
Thus, the lower bound (\ref{xiang bound}) can be made arbitrarily close to zero within the class of invertible $\mathbf{H}$-matrices.}
\end{ex}

For $A\circ\mathcal{M}(A)^{-1}$, however, the corresponding $\mathbf{M}$-matrix estimate yields a dimension-dependent positive lower bound. More precisely, Xiang \cite{xiang2003hadamard} established for an invertible $\mathbf{M}$-matrix $M$ that
\[
q(M\circ M^{-1})
\geq
\max\left\{
1-\rho(J_M)^2,\,
\frac{1+\rho(J_M)^{1/(n+2)}}
     {1+(n-1)\rho(J_M)^{1/(n+2)}}
\right\}.
\]
Combining this with \eqref{compareq*andq} gives the following result.

\begin{thm}\label{h matrix jacobi bound}
{\rm Let $A$ be an invertible $\mathbf{H}$-matrix of order $n$. Then
\[
q\bigl(A\circ\mathcal{M}(A)^{-1}\bigr)
\geq
\max\left\{
1-\rho(J_{\mathcal{M}(A)})^2,\,
\frac{1+\rho(J_{\mathcal{M}(A)})^{1/(n+2)}}
     {1+(n-1)\rho(J_{\mathcal{M}(A)})^{1/(n+2)}}
\right\}.
\]}
\end{thm}

We now give a necessary condition for the bound (\ref{xiang bound}) for $q(A\circ A^{-1})$ to approach zero. This result parallels the limiting behavior considered earlier for invertible $\mathbf{M}$-matrices (Section \ref{jacobi m}). There, approaching the sharp lower bound $\frac{2}{n}$ forces the spectral radius of the Jacobi iteration matrix to approach $1$. Here, for invertible $\mathbf{H}$-matrices, the limiting value is zero: if
$q(A_k\circ A_k^{-1})\to0$, then
$\rho(J_{\mathcal{M}(A_k)})\to1$.

\begin{pro}\label{prop:H-q-zero-jacobi}
{\rm Let $\{A_k\}$ be a sequence of invertible $\mathbf{H}$-matrices of order $n$. If 
$
q(A_k\circ A_k^{-1})\to0,
$
then
$
\rho(J_{\mathcal{M}(A_k)})\to1.
$}
\end{pro}

\begin{proof}
From (\ref{xiang bound}), for each $k$, we have
\[
0\leq
\frac{1-\rho(J_{\mathcal{M}(A_k)})^2}
     {1+\rho(J_{\mathcal{M}(A_k)})^2}
\leq
q(A_k\circ A_k^{-1}).
\]
Therefore, $q(A_k\circ A_k^{-1})\to0$ implies $ 
\dfrac{1-\rho(J_{\mathcal{M}(A_k)})^2}
     {1+\rho(J_{\mathcal{M}(A_k)})^2}
\to0.$
Since $\mathcal{M}(A_k)$ is an invertible $\mathbf{M}$-matrix,
$0\leq\rho(J_{\mathcal{M}(A_k)})<1$ for every $k$. Thus
$\rho(J_{\mathcal{M}(A_k)})\to1$.
\end{proof}

It is now natural to ask whether the converse implication holds, that is does 
$\rho(J_{\mathcal{M}(A_k)})\to1
$ imply $q(A_k\circ A_k^{-1})\to0?
 $
 However, for the simple sequence of matrices $A_k=\begin{pmatrix}
     1 & 0 \\ 0 & \frac{1}{k}
\end{pmatrix}$, we have $\rho(J_{\mathcal{M}(A_k)})=1$ but $q(A_k\circ A_k^{-1})=1\nrightarrow 0$ as $k \to \infty$.

\subsubsection{Upper spectral bounds and positive diagonal symmetrizability}\label{subsubsec:h-upper-bounds}
We now turn to upper bounds for $q(A\circ A^{-1})$. For invertible $\mathbf{M}$-matrices, the classical upper bound is $1$. It is therefore natural to ask whether the same bound continues to hold for invertible $\mathbf{H}$-matrices. We first show that the answer is affirmative in dimension two.

\begin{pro}\label{prop:h-2by2-upper}
{\rm Let $A\in\mathbb{R}^{2\times2}$ be an invertible $\mathbf{H}$-matrix. Then $
 0<q(A\circ A^{-1})\leq1.
$}
\end{pro}

\begin{proof}
Let
$
A=\begin{pmatrix}a&b\\ c&d\end{pmatrix}.
$
Since $A$ is an invertible $\mathbf{H}$-matrix, its comparison matrix is an invertible $\mathbf{M}$-matrix, and hence $|ad|>|bc|$. Moreover,
$
A^{-1}
=\frac{1}{ad-bc}
\begin{pmatrix}
\ \ d&-b\\
-c&\ \ a
\end{pmatrix},
$
so $A\circ A^{-1}
=
\frac{1}{ad-bc}
\begin{pmatrix}
ad&-b^2\\
-c^2&ad
\end{pmatrix}.
$
The eigenvalues of $A\circ A^{-1}$ are
$
1
$
and
$
\frac{ad+bc}{ad-bc}.
$
Thus, $1\in\sigma(A\circ A^{-1})$, and therefore
$q(A\circ A^{-1})\leq1$. Since $A\circ A^{-1}$ is invertible, we also have
$q(A\circ A^{-1})>0$.
\end{proof}

For matrices of arbitrary order, we obtain the same upper bound under positive diagonal symmetrizability. This assumption is motivated by the connection, discussed in the Introduction, between positive diagonal symmetrizability and the equality case in the upper bound $q(A\circ A^{-1})\leq1$ for invertible $\mathbf{M}$-matrices.

Recall that a real matrix $A$ is said to be {\it diagonally symmetrizable} if there exists an invertible diagonal matrix $D$ such that $AD$ is symmetric. If $D$ can be chosen to have positive diagonal entries, then $A$ is called {\it positive diagonally symmetrizable} (PDS). Equivalently, if $D=\operatorname{diag}(d_1,\ldots,d_n)$ with $d_i>0$, then $A$ is PDS if
$
a_{ij}d_j=a_{ji}d_i,$ $1\leq i,j\leq n.
$

In particular, every real symmetric matrix is PDS by taking $D=I$. We now show that the upper bound $1$ holds for PDS invertible $\mathbf{H}$-matrices of arbitrary order.

\begin{thm}\label{thm:PDS-H-upper}
{\rm Let $A$ be a positive diagonally symmetrizable invertible $\mathbf{H}$-matrix. Then
$
1\in\sigma(A\circ A^{-1}),
$ and consequently
$
0<q(A\circ A^{-1})\leq1.
$}

\end{thm}

\begin{proof}
Since $A$ is PDS, there exists a positive diagonal matrix $D=\operatorname{diag}(d_1,\ldots,d_n)$ such that $AD$ is symmetric. Writing $A^{-1}=(\alpha_{ij})$, the symmetry of $(AD)^{-1}=D^{-1}A^{-1}$ gives
$
\frac{\alpha_{ij}}{d_i}=\frac{\alpha_{ji}}{d_j},
$
and hence
$
\alpha_{ij}d_j=\alpha_{ji}d_i.
$
Let $C=A\circ A^{-1}$ and $d=(d_1,\ldots,d_n)^T$. Then

$$
(Cd)_i=\sum_{j=1}^n a_{ij}\alpha_{ij}d_j
=d_i\sum_{j=1}^n a_{ij}\alpha_{ji}
=d_i(AA^{-1})_{ii}
=d_i.
$$
Thus $Cd=d$, so $1\in\sigma(C)$ and therefore $q(C)\leq1$. Since $C=A\circ A^{-1}$ is an invertible $\mathbf{H}$-matrix, $q(C)>0$. Consequently,
$
0<q(A\circ A^{-1})\leq1.
$
\end{proof}

We now examine when equality occurs in the upper bound $q(A\circ A^{-1})\leq 1$ for PDS invertible $\mathbf{H}$-matrices, leading to a complete characterization in the irreducible case.

\begin{lem}\label{lem:symmetric-H-eigenvalue}
{\rm Let $B=(b_{ij})$ be a real symmetric invertible $\mathbf H$-matrix with $b_{ii}>0$ for all $i$. Then $
\lambda_{\min}(B)\geq\lambda_{\min}(\mathcal M(B))>0.
$ In particular, $B$ is positive definite.}
\end{lem}

\begin{proof}
Since $B$ is symmetric, $\mathcal M(B)$ is symmetric. Moreover, since $B$ is an invertible $\mathbf H$-matrix, $\mathcal M(B)$ is a nonsingular $\mathbf M$-matrix. Hence $\mathcal M(B)$ is positive definite, so $\lambda_{\min}(\mathcal M(B))>0$. For any $x\in\mathbb R^n$, since $b_{ii}>0$,
$$
\begin{aligned}
x^TBx
&=\sum_i b_{ii}x_i^2+2\sum_{i<j}b_{ij}x_ix_j
\geq\sum_i |b_{ii}|x_i^2-2\sum_{i<j}|b_{ij}|\,|x_i|\,|x_j|\\
&=|x|^T\mathcal M(B)|x|
\geq\lambda_{\min}(\mathcal M(B))\|x\|_2^2.
\end{aligned}
$$
Taking the minimum over $|x|_2=1$ gives
$
\lambda_{\min}(B)\geq\lambda_{\min}(\mathcal M(B))>0.
$
\end{proof}

\begin{pro}\label{prop:PDS-H-equality}
{\rm Let $A=(a_{ij})$ be a positive diagonally symmetrizable invertible
$\mathbf{H}$-matrix. If the diagonal entries of $A$ have the same sign,
then
$
q(A\circ A^{-1})=1.
$}
\end{pro}

\begin{proof}
Since $A$ is positive diagonally symmetrizable, there exists a positive
diagonal matrix $D=\operatorname{diag}(d_1,\ldots,d_n)$ such that $AD$
is symmetric. Set$\colon$
$
S=D^{-1/2}AD^{1/2}.
$
Since $AD=DA^T$, we have
\[
S^T
=D^{1/2}A^TD^{-1/2}
=D^{-1/2}AD^{1/2}
=S.
\]
Thus $S$ is symmetric. Moreover, $S$ is diagonally similar to $A$, so
$S$ is an invertible $\mathbf{H}$-matrix, and $
s_{ii}=a_{ii}$, $i=1,\ldots,n.
$

Suppose first that $a_{ii}>0$ for all $i$. Then $S$ is a symmetric invertible $\mathbf{H}$-matrix with positive diagonal entries.  
Therefore, by Lemma \ref{lem:symmetric-H-eigenvalue}, $S$ is positive definite.
If $a_{ii}<0$ for all $i$, then $-S$ is a symmetric invertible
$\mathbf{H}$-matrix with positive diagonal entries, and the same
argument shows that $-S$ is positive definite. Thus, in either case,
there exists a real symmetric positive definite matrix $P$, namely
$P=S$ or $P=-S$, such that $S\circ S^{-1}=P\circ P^{-1}.$

For a real symmetric positive definite matrix $P$, the classical Hadamard-product inequality gives $P\circ P^{-1}\succeq I.$ Hence every eigenvalue of $P\circ P^{-1}$ is at least $1$. On the
other hand, since $P$ and $P^{-1}$ are symmetric,
\[
\begin{aligned}
\bigl((P\circ P^{-1})\mathbf{1}\bigr)_i
&=\sum_{j=1}^n p_{ij}(P^{-1})_{ij}\\
&=\sum_{j=1}^n p_{ij}(P^{-1})_{ji}\\
&=(PP^{-1})_{ii}=1.
\end{aligned}
\]
Thus $(P\circ P^{-1})\mathbf{1}=\mathbf{1},$ so $1\in\sigma(P\circ P^{-1})$. Consequently, $q(P\circ P^{-1})=1.$

Finally, $S^{-1}=D^{-1/2}A^{-1}D^{1/2},$ and hence $S\circ S^{-1}=D^{-1}(A\circ A^{-1})D.$
Thus $S\circ S^{-1}$ and $A\circ A^{-1}$ are similar and have the
same spectrum. Therefore
\[
q(A\circ A^{-1})
=q(S\circ S^{-1})
=q(P\circ P^{-1})
=1.\qedhere
\]
\end{proof}

The converse of Proposition~\ref{prop:PDS-H-equality} does not hold in general. Indeed, for the reducible matrix
$A=\operatorname{diag}(1,-1)$, which is a PDS invertible $\mathbf H$-matrix, we have
$A\circ A^{-1}=-I$ and hence $q(A\circ A^{-1})=1$, although the diagonal entries of $A$ have different signs. This example naturally leads us to consider the irreducible case. Recall that an $n\times n$ matrix $A$ is said to be \textit{reducible} if it is permutationally similar to a matrix of the form
$\begin{pmatrix} B&0\\ C&D \end{pmatrix}$,
where $B$ and $D$ are square matrices, or if $n=1$ and $A=0$. Otherwise, $A$ is said to be \textit{irreducible}. Remarkably, irreducibility rules out the behavior illustrated above and yields a complete characterization of the equality case.

\begin{thm}\label{thm:PDS-H-equality-characterization}
{\rm Let $A=(a_{ij})$ be an irreducible positive diagonally symmetrizable
invertible $\mathbf{H}$-matrix. Then
\[
q(A\circ A^{-1})=1
\quad\Longleftrightarrow\quad
a_{11},\ldots,a_{nn}\text{ have the same sign}.
\]}
\end{thm}
\begin{proof}
The sufficiency follows directly from Proposition~\ref{prop:PDS-H-equality}. We prove the necessity by contradiction.
Suppose that $q(A\circ A^{-1})=1$, but the diagonal entries of $A$ have mixed signs.

Since $A$ is positive diagonally symmetrizable, there exists a positive diagonal matrix $D$ such that $AD$ is symmetric, and set $S=D^{-1/2}AD^{1/2}$. Since $AD=DA^T$,
$$
S^T=D^{1/2}A^TD^{-1/2}=D^{-1/2}AD^{1/2}=S.
$$

Moreover, $\mathcal M(S)=D^{-1/2}\mathcal M(A)D^{1/2}$, so $S$ is an invertible $\mathbf H$-matrix. Positive diagonal similarity preserves irreducibility and the diagonal entries; hence $S$ is irreducible and also has mixed signs on its diagonal.
Since $S^{-1}=D^{-1/2}A^{-1}D^{1/2}$, we have
$$
S\circ S^{-1}=D^{-1}(A\circ A^{-1})D,
\qquad
q(S\circ S^{-1})=q(A\circ A^{-1})=1.
$$ 
After a simultaneous permutation of rows and columns, write
$$
S=\begin{pmatrix}P&\ \ E\\ E^T&-N\end{pmatrix},
$$
where both diagonal blocks are nonempty and $P,N$ have positive diagonal entries. Since $\mathcal M(P)$ and $\mathcal M(N)=\mathcal M(-N)$ are principal submatrices of the nonsingular $\mathbf M$-matrix $\mathcal M(S)$, they are nonsingular $\mathbf M$-matrices. Thus $P$ and $N$ are symmetric invertible $\mathbf H$-matrices with positive diagonal entries. By Lemma~\ref{lem:symmetric-H-eigenvalue},
$
P\succ0$ and $N\succ0.
$
Also, $E\neq0$, since otherwise $S$ would be reducible.

Set $
C=S\circ S^{-1},
$ and $
K=N+E^TP^{-1}E.
$
Since $P,N\succ0$, we have
$
K\succ0,$ $ P+EN^{-1}E^T\succ0.
$
Using the Schur complements of $P$ and $-N$, the block inverse formula gives
$$
S^{-1}=
\begin{pmatrix}
(P+EN^{-1}E^T)^{-1}&P^{-1}EK^{-1}\\
K^{-1}E^TP^{-1}&-K^{-1}
\end{pmatrix}.
$$

Hence the upper-left block of $S^{-1}$ is positive definite, whereas the lower-right block is negative definite. Thus the diagonal entries of $S^{-1}$ have the same signs as the corresponding diagonal entries of $S$. Consequently,
$$
c_{ii}=s_{ii}(S^{-1})_{ii}>0\qquad(i=1,\ldots,n).
$$

By the previously established result on Hadamard products involving inverses of invertible $\mathbf H$-matrices, $C=S\circ S^{-1}$ is an invertible $\mathbf H$-matrix. Since $S$ and $S^{-1}$ are symmetric, $C$ is symmetric. Therefore, by Lemma~\ref{lem:symmetric-H-eigenvalue}, $
C\succ0.
$ It follows that $q(C)=\lambda_{\min}(C)$. Since $q(C)=1$, we obtain
$
\lambda_{\min}(C)=1.
$

Symmetry of $S^{-1}$ gives
$$
(C\mathbf1)_i
=\sum_j s_{ij}(S^{-1})_{ij}
=\sum_j s_{ij}(S^{-1})_{ji}
=(SS^{-1})_{ii}=1.
$$
Thus
$
C\mathbf1=\mathbf1
$ and $
\mathbf1^TC\mathbf1=n.
$

Let $
z=\begin{pmatrix}\ \ \mathbf1\\-\mathbf1\end{pmatrix},
$ where the two blocks correspond to the positive and negative diagonal entries of $S$. Writing $C$  as
$$
C=\begin{pmatrix}C_{11}&C_{12}\\ C_{12}^T&C_{22}\end{pmatrix},
$$ we obtain $
z^TCz=n-4 (\mathbf1^TC_{12}\mathbf1).
$ From the block inverse formula, $
(S^{-1})_{12}=P^{-1}EK^{-1},
$ and therefore
$
C_{12}=E\circ(P^{-1}EK^{-1}).
$ 
 Hence
\[
\begin{aligned}
\mathbf1^TC_{12}\mathbf1
&=\sum_{i,j}(C_{12})_{ij}
=\sum_{i,j}E_{ij}(P^{-1}EK^{-1})_{ij}\\
&=\operatorname{tr}(E^TP^{-1}EK^{-1})
=\operatorname{tr}(E^TP^{-1}EK^{-1})\\
&=\operatorname{tr}\!\left(
K^{-1/2}E^TP^{-1}EK^{-1/2}
\right)\\
&=\operatorname{tr}\!\left[
(P^{-1/2}EK^{-1/2})^T
(P^{-1/2}EK^{-1/2})
\right]\\
&=\|P^{-1/2}EK^{-1/2}\|_F^2>0.
\end{aligned}
\]
Since $P,K\succ0$, the matrices $P^{-1/2}$ and $K^{-1/2}$ are nonsingular. Thus $E\neq0$ implies
$P^{-1/2}EK^{-1/2}\neq0$, which gives the strict inequality.
 
Since $z^Tz=n$, it follows that
$$
\frac{z^TCz}{z^Tz}
=1-\frac{4}{n}\|P^{-1/2}EK^{-1/2}\|_F^2<1.
$$

This contradicts $\lambda_{\min}(C)=1$, since the Rayleigh principle implies
$$
\frac{x^TCx}{x^Tx}\geq\lambda_{\min}(C)=1
$$
for every nonzero $x$. Therefore the diagonal entries of $A$ cannot have mixed signs. Since an invertible $\mathbf H$-matrix has nonzero diagonal entries, they must all have the same sign. This completes the proof.
\end{proof}

\section{Conclusion} \label{sec:conclusion}
In this paper, we have investigated spectral bounds for Hadamard products involving invertible $\mathbf{M}$-matrices, invertible $\mathbf{H}$-matrices, and their inverses. For invertible $\mathbf{M}$-matrices, we revisited the lower bound $\frac{2}{n}$ for $q(A\circ A^{-1})$ and completed the discussion of its sharpness. We also studied its relation to the Jacobi iteration matrix. In particular, while $q(A_k\circ A_k^{-1})\to\frac{2}{n}$ implies $\rho(J_{A_k})\to1$, the converse does not hold in general, and we identified a class for which it does hold.

We then extended the investigation to invertible $\mathbf{H}$-matrices. Although $A\circ A^{-1}$ remains an invertible $\mathbf{H}$-matrix, its minimum eigenvalue modulus can be arbitrarily close to zero. Thus, the positive dimension-dependent lower bound available for invertible $\mathbf{M}$-matrices does not extend to the full class of invertible $\mathbf{H}$-matrices. However, by replacing $A^{-1}$ with $\mathcal{M}(A)^{-1}$, meaningful lower bounds can be recovered. This yields bounds for $q(A\circ\mathcal{M}(A)^{-1})$, including bounds involving the Jacobi iteration matrix of the comparison matrix. We also established that $q(A_k\circ A_k^{-1})\to0$ implies $\rho(J_{\mathcal{M}(A_k)})\to1$.

Finally, motivated by the connection between Hadamard products and positive diagonal symmetrizability, we considered upper bounds in the $\mathbf{H}$-matrix setting. We showed that if $A$ is a positive diagonally symmetrizable invertible $\mathbf{H}$-matrix, then $1\in\sigma(A\circ A^{-1})$, and consequently $q(A\circ A^{-1})\leq1$. Moreover, in the irreducible case, equality holds if and only if the diagonal entries of $A$ have the same sign.

It is worth noting that all the lower-bound results obtained in this work remain valid when $A$ is a complex $\mathbf{H}$-matrix. The situation is different, however, for the upper bounds, where the underlying field plays an essential role. While the inequality $q(A\circ A^{-1})\leq1$ remains an open question for general invertible $\mathbf{H}$-matrices with real entries, it does not hold in general over the complex field. For example, $A=\begin{pmatrix}3&-1+i&-1+i\\-i&3&1+i\\-1-i&-i&3\end{pmatrix}$ is a complex invertible $\mathbf{H}$-matrix for which $q(A\circ A^{-1})\approx1.09552>1$. These results highlight both the similarities and the essential differences between the $\mathbf{M}$- and $\mathbf{H}$-matrix settings and lead naturally to several further questions.
\subsection*{Open Problems and Future Directions}\label{subsec:open-problems}
\addcontentsline{toc}{subsection}{Open Problems and Future Directions} 

The results of this paper suggest several natural questions for further investigation. One direction is to characterize subclasses of invertible $\mathbf H$-matrices for which $q(A\circ A^{-1})$ admits a positive lower bound and to determine the corresponding sharp bounds. It would also be interesting to obtain sharper estimates for $q(A\circ\mathcal M(A)^{-1})$ under additional structural assumptions on $A$.

Another natural question concerns the upper bound. In view of the complex counterexample discussed above, the problem is specific to the real setting: does every real invertible $\mathbf H$-matrix $A$ satisfy $q(A\circ A^{-1})\leq1$? If not, it would be interesting to characterize subclasses of real invertible $\mathbf H$-matrices for which this inequality holds.

Several questions also arise from the bounds involving Jacobi iteration matrices. In particular, one may characterize classes of invertible $\mathbf M$-matrices for which
$\rho(J_{A_k})\to1$
implies
$q(A_k\circ A_k^{-1})\to\frac{2}{n}$,
as well as classes of sequences of invertible $\mathbf H$-matrices for which
$\rho(J_{\mathcal M(A_k)})\to1$
implies
$q(A_k\circ A_k^{-1})\to0$.
Resolving these converse implications would further clarify the role of the Jacobi iteration matrix in these spectral bounds.

\vspace{0.5cm}

\section*{Acknowledgements}
\addcontentsline{toc}{section}{Acknowledgements} The authors thank Prof. K. C. Sivakumar for introducing this problem to them. The work of PIMS Postdoctoral Fellow S.\ Mondal leading to this publication was supported in part by the Pacific Institute for the Mathematical Sciences.

\end{document}